\documentclass[a4paper, 12pt]{article}
\usepackage{amsfonts}
\usepackage{amssymb}
\usepackage{amsthm}
\usepackage{amsmath}
\usepackage{mathtext}
\usepackage[utf8x]{inputenc}
\usepackage[english]{babel}
\usepackage{color}
\usepackage{euscript,amssymb,amsfonts,amsmath,amsthm,graphicx}
\usepackage{mathrsfs}
\usepackage[left=3cm,right=1.5cm,top=2cm,bottom=2cm,bindingoffset=0cm]{geometry}
\usepackage{enumerate}

\newtheorem{theorem}{\indent \sc Theorem}

\newtheorem{cor}{\indent \sc Corollary}

\newcommand{\eps}{\varepsilon}

\newcommand{\R}{\mathbb R}
\newcommand{\N}{\mathbb N}
\newcommand{\E}{{\sf E}}
\newcommand{\D}{{\sf D}}
\newcommand{\Prob}{{\sf P}}

\renewcommand{\le}{\leqslant}
\renewcommand{\ge}{\geqslant}
\newcommand{\I}{{\bf1}}

\newcommand{\eqd}{\stackrel{d}{=}}

\newcommand{\abs}[1]{\left|#1\right|}

\newcommand{\G}{\mathcal G}

\newcommand{\qt}{\sigma^2}
\newcommand{\gopt}{\gamma_*}
\newcommand{\scE}{{\textsc{e}}}
\newcommand{\scR}{{\textsc{r}}}
\newcommand{\es}{L_{\scE,n}}
\newcommand{\roz}{L_{\scR,n}}
\newcommand{\essC}{A_\scE}
\newcommand{\rozC}{A_\scR}
\newcommand{\CE}{C_\scE}
\newcommand{\CR}{C_\scR}

\newcommand{\osiM}{\Lambda_n}
\newcommand{\essM}{M_n}
\usepackage{multirow}

\date{}

\title{On natural convergence rate estimates in the Lindeberg theorem\thanks{Research on Theorem~\ref{ThAWanalog} was supported by the Russian Science Foundation (Project No.\,18-11-00155). The rest part of the paper was supported by the Russian Foundation for Basic Research (project 19-07-01220-a) and by the Ministry for Education and Science of Russia (grant No.~MD--189.2019.1).}}

\author{Ruslan Gabdullin\thanks{Lomonosov Moscow State University, Moscow, Russia}, Vladimir Makarenko\thanks{Lomonosov Moscow State University, Moscow, Russia}, and Irina Shevtsova\footnote{Corresponding author. Hangzhou Dianzi University, Hangzhou, China. Lomonosov Moscow State University, Moscow, Russia. Institute of Informatics Problems of Federal Research Center ``Computer Science and Control,'' Russian Academy of Sciences, Moscow, Russia. E-mail: ishevtsova@cs.msu.ru}}

\begin{document}

\maketitle

\begin{abstract}
We prove two estimates of the rate of convergence in the Lindeberg theorem, involving algebraic truncated third-order moments and the classical Lindeberg fraction, which generalize a series of inequalities due to  (Esseen, 1969), (Rozovskii, 1974), (Wang, Ahmad, 2016), some of our recent results in (Gabdullin, Makarenko, Shevtsova, 2018, 2019) and, up to constant factors, also (Katz, 1963), (Petrov, 1965), (Osipov, 1966). The technique used in the proof is completely different from that in (Wang, Ahmad, 2016) and is based on some extremal properties of introduced fractions which has not been noted  in (Katz, 1963), (Petrov, 1965), (Wang, Ahmad, 2016).
\end{abstract}


\section{Introduction and notation}

Let $X_1, X_2, \ldots, X_n$ be independent random variables (r.v.'s) with distribution functions (d.f.'s) $F_k(x) = \Prob(X_k < x)$, $x\in\R,$ expectations $\E X_k = 0$, and variances $\sigma_k^2 = \D X_k$, $k=1,\ldots,n$, such that
$$
B_n^2:= \sum_{k=1}^n \sigma_k^2>0.
$$
Let us denote
$$
S_n = X_1 + X_2 + \cdots + X_n, \quad 
\widetilde S_n = \frac{S_n-\E S_n}{\sqrt{\D S_n}} = \sum_{k=1}^n \frac{X_k}{B_n},
$$
$$
\overline{F}_n(x) = \Prob(\widetilde S_n < x), \quad \Phi(x) = \frac{1}{\sqrt{2\pi}} \int_{-\infty}^x e^{-t^2/2}\,dt, \quad x \in \R,
$$
$$
\Delta_n= \sup_{x \in \R} \abs{\overline{F}_n(x) - \Phi(x)},
$$
$$
\sigma_k^2(z) = \E X_k^2 \I(\abs{X_k} \ge z), \quad \mu_k(z) = \E X_k^3 \I(\abs{X_k} < z),\quad k=1,\ldots,n
$$
$$
L_n(z) = \frac{1}{B_n^2} \sum_{k=1}^n \sigma_k^2(zB_n) = \frac{1}{B_n^2} \sum_{k=1}^n \E X_k^2\I (\abs{X_k} \ge zB_n), \quad z \ge 0,
$$
so that $L_n(0)=1$. The function $L_n(z)$ is called the \textit{Lindeberg fraction}. In the case where $X_1,\ldots,X_n$ are independent and identically distributed (i.i.d.) we shall denote their common distribution function by $F$. 

Let  $\G$ be a set of all increasing functions $g\colon(0,\infty)\to(0,\infty)$ such that the function $z / g(z)$ is also increasing for $z>0$ (for convenience, here and in what follows, we use the terms ``increasing'' and ``decreasing'' in a wide sense, i.\,e. ``non-decreasing'', ``non-increasing''). Where appears, the value $g(0)$ is assumed to be an arbitrary non-negative number. The class~$\G$ was originally introduced by Katz~\cite{Katz1963} and further used in~\cite{Petrov1965,KorolevPopov2012,KorolevDorofeyeva2017,GabdullinMakarenkoShevtsova2019}. In~\cite{GabdullinMakarenkoShevtsova2019} it was proved that:  
\begin{itemize}
\item[\rm(i)] For every $g \in \G$ and $a > 0$ we have
\begin{equation}\label{gmin_g_gmax_ineq}
g_0(z, a):=\min\left\{\frac{z}{a}, 1\right\}\ \le\ \frac{g(z)}{g(a)}\ \le\ \max\left\{\frac{z}{a}, 1\right\}:= g_1(z, a),\quad z>0,
\end{equation}
moreover $g_0(\,\cdot\,, a),$ $g_1(\,\cdot\,, a) \in \G.$
\item[\rm(ii)] Every function from $\G$ is continuous on $(0,\infty);$
\end{itemize}
Property~\eqref{gmin_g_gmax_ineq} means that, asymptotically, every function $g\in\G$ is between a constant and a linear function as its argument goes to infinity. For example, besides $g_0$, $g_1$, the class $\G$ also includes the following functions:
$$
g_{\textsc{c}}(z)\equiv 1,\quad g_*(z)=z,\quad c \cdot z^{\delta}, \quad c\cdot g(z),\qquad z>0,
$$
for every $c > 0$, $\delta \in [0,1]$, and $g\in\G$.

For every function $g\in\G$ such that $\E X_k^2g(\abs{X_k}) < \infty,$ $k=1,\ldots,n$, Katz~\cite{Katz1963} (in 1963, for identically distributed random summands) and Petrov~\cite{Petrov1965} (in 1965, in the general situation) proved that
\begin{equation} \label{Katz_Petrov_ineq}
\Delta_n \le \frac{A_1}{B_n^2g(B_n)}\sum_{k=1}^n \E X_k^2g(\abs{X_k}), 
\end{equation}
$A_1$ being a universal constant whose best known upper bound $A_1\le1.87$ is due to Korolev and Dorofeyeva~\cite{KorolevDorofeyeva2017}. A few years earlier, letting $n=1$ and using results of~\cite{Agnew1957} (see also~\cite{BhattacharyaRangaRao1982}) and~\cite{KondrikMikhailovChebotarev2006} (see also~\cite{ChebotarevKondrikMikhailov2007}),  Korolev and Popov~\cite{KorolevPopov2012} established a lower bound
$$
A_1\ge \sup_{x>0}\abs{\frac1{1+x^2}-\Phi(-x)}=0.54093\ldots,
$$
which remains the best known one until now. For the sake of unambiguity, here and in what follows,  by constants appearing in majorizing expressions we mean their least possible values guaranteeing the validity of the corresponding inequalities for all parameters under consideration.

Inequality~\eqref{Katz_Petrov_ineq} with $g(x) = \min \{x/B_n, 1\}$ yields
\begin{equation} \label{Osipov_eps1_ineq}
\Delta_n \le A_2 \cdot (L_n(1) + \osiM(1))
\end{equation}
where $A_2\le A_1$ and
$$
\osiM(\eps) :=  \frac{1}{B_n^3}\sum_{k=1}^n \E|X_k|^3\I(|X_k|<\eps B_n) =  \frac{1}{B_n^3}\sum_{k=1}^n \int_{\abs{x} < \eps B_n}\abs{x}^3\,dF_k(x),\quad \eps>0.
$$
As it was pointed out in~\cite{Loh1975,Paditz1984}, 
$$
\E|X|^3\I(|X|<1)+ \E X^2\I(|X|\ge1)= \E X^2\min\{ 1,|X|\}(\I(X\in B) +\I(X\notin B))
$$
$$
\le \E|X|^3\I(X \in B) + \E X^2\I(X\notin B)
$$
for any Borel $B\subseteq\R$, so that
$$
L_n(1) + \osiM(1) = \inf_{\eps > 0} (L_n(\eps) + \osiM(\eps)),
$$
and, hence, \eqref{Osipov_eps1_ineq} yields the inequality 
\begin{equation} \label{Osipov_ineq}
\Delta_n \le A_2 \cdot  (L_n(\eps) + \osiM(\eps)) \quad\text{for every }\eps>0.
\end{equation}
Inequality~\eqref{Osipov_ineq} was proved by Osipov~\cite{Osipov1966} in 1966, independently of~\eqref{Katz_Petrov_ineq}. Since 
\begin{equation}\label{Lambda(eps)<=eps}
\osiM(\eps)\le\frac{\eps}{B_n^2}\sum_{k=1}^n\E X_k^2\I(|X_k|<\eps B_n)\le\eps,
\end{equation}
inequality~\eqref{Osipov_ineq} trivially yields the Lindeberg theorem~\cite{Lindeberg1922} which states the sufficiency of the Lindeberg condition
$$
\lim_{n\to\infty}L_n(\eps)=0\quad\text{for all }\eps>0\eqno (L)
$$
for the validity of the CLT 
$$
\lim\limits_{n\to\infty}\Delta_n=0.\eqno ({CLT})
$$
On the other hand, according to Feller's theorem~\cite{Feller1935}, for the asymptotically negligible random summands, i.e. satisfying the Feller condition
$$
\lim_{n\to\infty}\max_{1\le k\le n}\frac{\sigma_k^2}{B_n^2}=0,\eqno (F)
$$
condition $(CLT)$ yields $(L)$. Taking into account also Feller's theorem, it is easy to conclude that the right- and the  left-hand sides of~\eqref{Osipov_ineq} either tend or do not tend to zero simultaneously, once the random summands are uniformly asymptotically negligible  in the sense of $(F)$. Thus, using Zolotarev's classification~\cite{Zolotarev1986}, inequality~\eqref{Osipov_ineq} can be called a \textit{natural} convergence rate estimate in the Lindeberg theorem.

In 1984 Paditz~\cite{Paditz1984} observed without proof that $A_1=A_2$; in  2012 Korolev and Popov~\cite{KorolevPopov2012} proved that $A_1 = A_2$. In other words, the function $g(x)=g_0(x,B_n)=\min\{x/B_n, 1\}$ minimizes the right-hand side of~\eqref{Katz_Petrov_ineq} (observe that this fact also trivially follows from property~\eqref{gmin_g_gmax_ineq} of the functions $g\in\mathcal G$ proved in~\cite{GabdullinMakarenkoShevtsova2019}). Hence, inequalities \eqref{Katz_Petrov_ineq}, \eqref{Osipov_eps1_ineq}, and \eqref{Osipov_ineq} are equivalent.

On the other hand, inequality~\eqref{Katz_Petrov_ineq} with $g(x)=x$ reduces to the celebrated Berry--Esseen inequality~\cite{Berry1941, Esseen1942} up to the constant factor $A_1$, for which Shevtsova~\cite{Shevtsova2013Inf} provides  an improved upper bound: $A_1\le0.4690$ in the  i.i.d. case and $A_1\le0.5583$ in the general situation.

In 1969 Esseen~\cite{Esseen1969}  managed to replace the \textit{absolute} truncated third order moments $\osiM(\eps)$ in Osipov's inequality~\eqref{Osipov_ineq} with  absolute values of the \textit{algebraic} ones and to prove that
\begin{equation} \label{Esseen_ineq}
\Delta_n \le \frac{A_3}{B_n^3} \sum_{k=1}^n \sup_{z > 0} \left\{\abs{\mu_k(z)}+z\sigma_k^2(z)\right\},
\end{equation}
with $A_3$ being an absolute constant. Moreover, in the same paper, by use of the traditional truncation techniques, Esseen provided a sketch of the proof of a bounded version of his inequality~\eqref{Esseen_ineq}:
\begin{equation} \label{Essen_limited_ineq}
\Delta_n \le \frac{A_4}{B_n^3} \sum_{k=1}^n \sup_{0 < z < B_n} \left\{\abs{\mu_k(z)}+z\sigma_k^2(z)\right\},
\end{equation}
which trivially yields~\eqref{Osipov_eps1_ineq} with $A_2\le A_4$, since $|\mu_k(z)|\le\E|X_k|^3\I(|X_k|<z)$ and the function 
$$
\E|X_k|^3\I(|X_k|<z)+z\sigma_k^2(z)=\E X_k^2\min\{|X_k|,z\}
$$
is monotonically increasing with respect to $z\ge0$ for every $k=1,\ldots,n$.  Hence,
$$
A_2\le A_4,\quad A_3\le A_4.
$$
Observe that, due to the left-continuity of the functions $\mu_k(z)$ and $\qt_k(z)$, $k=1,\ldots,n,$ for $z>0$,
the least upper bound over $z\in(0,B_n)$ in~\eqref{Essen_limited_ineq}  can be
replaced by the one over the set $z\in(0,B_n].$ In the i.i.d. case,  Esseen's inequality~\eqref{Essen_limited_ineq} takes the form
$$
\Delta_n\le \frac{A_4}{\sigma_1^3\sqrt{n}}\sup_{0<z\le \sigma_1\sqrt{n}} \left\{\abs{\mu_1(z)} + z\qt_1(z)\right\}
$$
and trivially yields the ``if'' part of Ibragimov's criteria~\cite{Ibragimov1966}, according to which, 
$\Delta_n=\mathcal O(n^{-1/2})$ as $n\to\infty$ if and only if
$$
\mu_1(z)=\mathcal O(1),\quad z\qt_1(z)=\mathcal O(1),\quad z\to\infty.
$$
The values of $A_3$ and $A_4$ remained unknown for a long time. Only in 2018 the present authors~\cite{GabdullinMakarenkoShevtsova2018} proved that $A_3\le2.66$ and $A_4\le2.73$.

In 1974 Rozovskii \cite{Rozovskii1974} proved that
\begin{equation} \label{Rozovskii_ineq}
\Delta_n \le \frac{A_5}{B_n^3} \sum_{k=1}^n \left( \abs{\mu_k(B_n)} + \sup_{0 < z < B_n} z \sigma_k^2(z) \right),
\end{equation}
where $A_5$ is an absolute constant whose value also remained unknown for a long time until the present authors deduced in~\cite{GabdullinMakarenkoShevtsova2018} that $A_5\le2.73$.
In~\cite[Section~5]{GabdullinMakarenkoShevtsova2018}, it is also shown that Esseen's and Rozovskii's fractions (right-hand sides of~\eqref{Essen_limited_ineq} and~\eqref{Rozovskii_ineq},  ignoring the constant factors $A_4$ and $A_5$) are incomparable even in the i.i.d. case, that is, Rozovskii's fraction may be less and may be greater than Esseen's fraction. 

Adopting ideas of Katz and Petrov~\cite{Katz1963,Petrov1965}, recently, Wang and Ahmad~\cite{WangAhmad2016} generalized Esseen's inequality~\eqref{Esseen_ineq} to
\begin{equation} \label{AW_classic_ineq}
\Delta_n \le \frac{A_6}{B_n^2g(B_n)}  \sum_{k=1}^n \sup_{z > 0}  \frac{g(z)}{z}  \Big(\abs{\mu_k(z)} + z\sigma_k^2(z) \Big), \quad g\in\G,
\end{equation}
where $A_6$ is an absolute constant whose value has not been given in~\cite{WangAhmad2016}. One can make sure that inequality~\eqref{AW_classic_ineq} trivially yields~\eqref{Katz_Petrov_ineq} with $A_1\le A_6$ (for the complete proof, see~\cite[p.\,648]{GabdullinMakarenkoShevtsova2019}) and with $g(z) = z$ reduces to~\eqref{Esseen_ineq} with $A_3\le A_6$. In~\cite{GabdullinMakarenkoShevtsova2019} it was shown that $A_6\le2.73$.
 
Inequalities~\eqref{Esseen_ineq},~\eqref{Essen_limited_ineq}, and~\eqref{Rozovskii_ineq} were improved and generalized in~\cite{GabdullinMakarenkoShevtsova2018} to
\begin{equation} \label{Esseen2018}
\Delta_n \le \frac{\essC(\eps,\gamma)}{B_n^3}  \sup_{0 < z < \eps B_n} \left\{ \gamma\abs{\sum_{k=1}^n \mu_k(z)} + z \sum_{k = 1}^n \sigma_k^2(z) \right\} =: \essC(\eps,\gamma) \cdot\es(\eps,\gamma),
\end{equation}
\begin{equation} \label{Rozovskii2018}
\Delta_n \le \frac{\rozC(\eps,\gamma)}{B_n^3}  \left( \gamma\abs{\sum_{k=1}^n \mu_k(\eps B_n)} + \sup_{0 < z < \eps B_n} z \sum_{k=1}^n \sigma_k^2(z)  \right) =: \rozC(\eps,\gamma) \cdot \roz(\eps,\gamma),
\end{equation}
for every $\eps>0$ and $\gamma>0$, where $\essC(\eps,\gamma),\,\rozC(\eps,\gamma)$ depend only on $\eps$ and $\gamma$, both are monotonically decreasing with respect to $\gamma>0$, and $\essC(\eps,\gamma)$ is also monotonically decreasing with respect to $\eps>0$. In particular, 
$$
A_4\le \essC(1,1) \le \essC(1,0.72) \le 2.73,
$$
$$
A_3\le \essC(+\infty,1)\le\max\{\essC(+\infty,0.97),\essC(4.35,1)\} \le 2.66,
$$
$$
A_5\le \rozC(1,1)\le \rozC(1,\gopt) \le 2.73,
$$
where
$$
\gopt=1/\sqrt{6\varkappa}=0.5599\ldots,\quad \varkappa=x^{-2}\sqrt{(\cos x-1+x^2/2)^2+(\sin x-x)^2}\Big|_{x=x_0}=0.5315\ldots,
$$
and $x_0=5.487414\ldots$ is the unique root of the equation
$$
8(\cos x - 1) + 8x\sin x - 4x^2\cos x - x^3\sin x = 0,\quad x\in(\pi,2\pi).
$$
Moreover, the functions $\essC(\eps,\gamma)$ and $\rozC(\eps,\gamma)$ are unbounded as $\eps\to0+$ for every $\gamma$, since $\lim\limits_{\eps\to0}\es(\eps,\gamma)=\lim\limits_{\eps\to0}\roz(\eps,\gamma)=0$. The question on boundedness of $\rozC(+\infty,\gamma)$ remains open. The values of $\essC(\eps,\gamma)$ and $\rozC(\eps,\gamma)$ for some other $\eps$ and $\gamma$ computed in~\cite{GabdullinMakarenkoShevtsova2018} are presented in tables~\ref{Tab:EsseenC1(L)} and~\ref{Tab:RozovskiiC1(L)}, respectively.  Observe that, in the i.i.d. case, the fractions $\es(1,1)$ and $\es(\infty,1)$ coincide with the Esseen fractions in~\eqref{Esseen_ineq} and~\eqref{Essen_limited_ineq}, respectively, and $\roz(1,1)$ coincides with the Rozovskii fraction in~\eqref{Rozovskii_ineq}, so that
$$
A_3=\essC(\infty,1),\quad A_4=\essC(1,1),\quad A_5=\rozC(1,1)\quad\text{in the i.i.d. case.}
$$

\begin{table}[h]
\begin{center}
\begin{tabular}{||c|c|c||c|c|c||c|c|c||}
\hline
$\eps$&$\gamma$&$\essC(\eps,\gamma)$& $\eps$&$\gamma$&$\essC(\eps,\gamma)$& $\eps$&$\gamma$&$\essC(\eps,\gamma)$
\\ \hline 
$1.21$&$0.2$&$2.8904$&$\infty$&$\gopt$&$2.6919$&$2.65$&$4$&$2.6500$\\ \hline 
$1.24$&$0.2$&$2.8900$&$1$&$0.72$&$2.7298$&$2.74$&$3$&$2.6500$\\ \hline 
$\infty$&$0.2$&$2.8846$&$1$&$\infty$&$2.7286$&$3.13$&$2$&$2.6500$\\ \hline 
$1.76$&$0.4$&$2.7360$&$4.35$&$1$&$2.6600$&$4$&$1.62$&$2.6500$\\ \hline 
$5.94$&$0.4$&$2.7300$&$\infty$&$1$&$2.6588$&$5.37$&$1.5$&$2.6500$\\ \hline 
$\infty$&$0.4$&$2.7299$&$\infty$&$0.97$&$2.6599$&$\infty$&$1.43$&$2.6500$\\ \hline 
$1$&$\gopt$&$2.7367$&$2.56$&$\infty$&$2.6500$&$\infty$&$\infty$&$2.6409$\\ \hline 
$1.87$&$\gopt$&$2.6999$&$2.62$&$5$&$2.6500$&$0+$&$\forall$&$\infty$\\ \hline 
\end{tabular}
\end{center}
\caption{Upper bounds for $\essC(\eps,\gamma)$ in~\eqref{Esseen2018}. Recall that $\gopt=0.5599\ldots$}
\label{Tab:EsseenC1(L)}
\end{table}

\begin{table}[h]
\begin{center}
\begin{tabular}{||c|c|c||c|c|c||}
\hline
$\eps$&$\gamma$&$\rozC(\eps,\gamma)$&
$\eps$&$\gamma$&$\rozC(\eps,\gamma)$
\\ \hline
$1.21$&$0.2$&$2.8700$&$1.99$&$\gopt$&$2.6600$\\ \hline 
$5.39$&$0.2$&$2.8635$&$2.12$&$\gopt$&$2.6593$\\ \hline 
$1.76$&$0.4$&$2.6999$&$3$&$\gopt$&$2.6769$\\ \hline 
$2.63$&$0.4$&$2.6933$&$5$&$\gopt$&$2.7562$\\ \hline 
$0.5$&$\gopt$&$3.0396$&$0+$&$\forall$&$\infty$\\ \hline 
$1$&$\gopt$&$2.7286$&&&\\ \hline 
\end{tabular}
\end{center}
\caption{Upper bounds for $\rozC(\eps,\gamma)$ in~\eqref{Rozovskii2018}. Recall that $\gopt=0.5599\ldots$}
\label{Tab:RozovskiiC1(L)}
\end{table}


Using the notation
$$
\essM(z):=\frac1{B_n^3}\sum_{k=1}^n \mu_k(zB_n) = \frac1{B_n^3}\sum_{k=1}^n \E X_k^3\I(|X_k|<zB_n),\quad z\ge0,
$$
one can rewrite inequalities~\eqref{Esseen2018} and~\eqref{Rozovskii2018} as
$$
\Delta_n \le \essC(\eps,\gamma)\sup_{0 < z < \eps} \left\{ \gamma\abs{\essM(z)} + zL_n(z) \right\} ,
$$
$$
\Delta_n \le \rozC(\eps,\gamma)\Big(\gamma\abs{\essM(\eps)} + \sup_{0 < z < \eps} zL_n(z) \Big).
$$

Furthermore, inequalities \eqref{AW_classic_ineq} and~\eqref{Esseen2018} with $\gamma=1$ were generalized in our previous paper~\cite{GabdullinMakarenkoShevtsova2019} to
\begin{multline} \label{EsseenAW2018}
\Delta_n \le \frac{\CE(\eps)}{B_n^2 g(B_n)} \cdot \sup_{0 < z < \eps B_n} \frac{g(z)}{z} \left( \abs{\sum_{k=1}^n \mu_k(z)} + z \sum_{k = 1}^n \sigma_k^2(z) \right)
\\
=\CE(\eps)\sup_{0 < z < \eps} \frac{g(zB_n)}{zg(B_n)} \left( \abs{\essM(z)} + zL_n(z) \right), \quad \eps>0, \quad g \in \G,
\end{multline}
with $\CE(\eps)\le\essC(\min\{1,\eps\},1)$,  $\CE(+0)=\infty,$ so that $A_6\le\CE(\infty)$ with the equality in the i.i.d. case, and inequality~\eqref{Rozovskii2018} with $\gamma=1$ was generalized in~\cite{GabdullinMakarenkoShevtsova2019-2} to
\begin{multline} \label{Rozovskii2019}
\Delta_n \le \frac{\CR (\eps)}{B_n^2 g(B_n)} \left( \frac{g(\eps B_n)}{\eps B_n} \abs{\sum_{k=1}^n \mu_k(\eps B_n)} +  \sup_{0 < z < \eps B_n} g(z)\sum_{k = 1}^n \sigma_k^2(z) \right)
\\
=\CR(\eps)\left(\frac{g(\eps B_n)}{\eps g(B_n)}  \abs{\essM(\eps)} + \sup_{0 < z < \eps}\frac{g(zB_n)}{g(B_n)}\,L_n(z) \right), \quad \eps>0, \quad g \in \G,
\end{multline}
with $\CR(\eps)\le\max\{1,\eps\}\cdot \rozC(\eps,1)$,  $\CR(+0)=\CR(\infty)=\infty$. It is easy to see that~\eqref{EsseenAW2018} and~\eqref{Rozovskii2019} with $g(z)=g_*(z)(\equiv z)$ reduce, respectively, to~\eqref{Esseen2018} and~\eqref{Rozovskii2018} with $\gamma=1$ and $\essC(\,\cdot\,,1)\le\CE(\,\cdot\,)$, $\rozC(\,\cdot\,,1)\le\CR(\,\cdot\,)$,
so that
$$
\CE(\eps)=\essC(\eps,1),\quad \CR(\eps)=\rozC(\eps,1)\quad\text{for } \eps\in(0,1].
$$

Though the constants $A_1=A_2\le1.87$ in Katz--Petrov's~\eqref{Katz_Petrov_ineq} and Osipov's~\eqref{Osipov_ineq} inequalities are more optimistic than $\essC(1,1)=\CE(1)\le2.73$, $\CR(1)=\rozC(1,1)\le2.73$, inequalities \eqref{Esseen2018}, \eqref{Rozovskii2018}, \eqref{EsseenAW2018}, and \eqref{Rozovskii2019} may be much sharper than~\eqref{Katz_Petrov_ineq} and \eqref{Osipov_ineq} due to the more favorable dependence of the appearing fractions on truncated third order moments $\abs{M_n(\,\cdot\,)}$, which vanishes, say, in the symmetric case, or for even $n$ and oscillating sequence $X_k\eqd(-1)^kX$, $k=1,\ldots,n,$ with one and the same r.v. $X$. In the above cases inequalities \eqref{Esseen2018}, \eqref{Rozovskii2018}  reduce to
$$
\Delta_n \le C_\eps\cdot\sup_{0 < z <\eps} zL_n(z),\quad C_\eps=\min\{\essC(\eps,\infty),\rozC(\eps,\infty)\}.
$$
Since $L_n(z)$ is non-increasing and left-continuous, the least upper bound here must be attained in an interior of the interval $(0,\eps)$:
$$
\Delta_n \le C_\eps\cdot z_nL_n(z_n)\quad\text{with some } z_n\in(0,\eps).
$$
The sequence $\{z_n\}_{n\in\N}$ here may be infinitesimal as $n\to\infty$. This follows from Ibragimov and Osipov's result~\cite{IbragimovOsipov1966} who proved that, in general, the estimate $\Delta_n\le C L_n(z)$ cannot hold  with a fixed $z>0$, even in the symmetric i.i.d. case.


\section{Motivation, Main Results and Discussion}

As it was noted before, the sum of truncated third order moments $\essM(\,\cdot\,)$ in~\eqref{EsseenAW2018} and~\eqref{Rozovskii2019} may be arbitrarily small or even vanish, so that the term depending on the Lindeberg fraction $L_n$ may be much more greater than the term  containing~$\essM$. Hence, it would be useful to have a possibility to balance the contribution of the terms $\abs{\essM}$ and $L_n$ to optimize the resulting bound. Similarly to~\eqref{Esseen2018} and~\eqref{Rozovskii2018}, let us introduce a balancing parameter $\gamma>0$ and for $\eps > 0$ and $g\in\G$ denote
\begin{multline} \label{L^3_def}
\es(g,\eps, \gamma) = \frac{1}{B_n^2 g(B_n)}  \sup_{0 < z < \eps B_n} \frac{g(z)}{z} \Bigg\{ \gamma \abs{\sum_{k = 1}^n \mu_k(z)} + z\sum_{k = 1}^n\sigma_k^2(z) \Bigg\} 
\\
= \sup_{0 < z < \eps} \frac{g(zB_n)}{zg(B_n)} \left( \gamma\abs{\essM(z)} + zL_n(z) \right),
\end{multline}
\begin{multline} \label{wave_L^3_def}
\roz(g,\eps, \gamma) = \frac{1}{B_n^2 g(B_n)}\left(\gamma\,\frac{g(\eps B_n)}{\eps B_n}\abs{\sum_{k=1}^n \mu_k(\eps B_n)}+\sup_{0<z < \eps B_n} g(z)\sum_{k=1}^n \sigma_k^2(z)\right) 
\\
= \gamma\,\frac{g(\eps B_n)}{\eps g(B_n)} \abs{\essM(\eps)} + \sup_{0 < z < \eps}\frac{g(zB_n)}{g(B_n)}\,L_n(z).
\end{multline}
Then $\es(g,\eps,1)$, $\roz(g,\eps,1)$ coincide with the corresponding fractions in the right-hand sides of~\eqref{EsseenAW2018}, \eqref{Rozovskii2019}; moreover, with 
$$
g(z)=g_*(z)\equiv z,\quad z>0,
$$
we have
$$
\es(g_*,\eps,\gamma) = \es(\eps,\gamma),\quad  \roz(g_*,\eps,\gamma) = \roz(\eps,\gamma), \quad \eps,\gamma>0.
$$
Moreover, $\es(\,\cdot\,,\eps,\,\cdot\,)$ is monotonically increasing with respect to $\eps>0$.

The main result of the present paper is the following 

\begin{theorem}\label{ThAWanalog}
For every $\eps > 0, \, \gamma > 0,$ and $g \in \G$ we have
\begin{equation} \label{AW_ineq}
\Delta_n \le \CE(\eps, \gamma)\cdot\es(g,\eps, \gamma),
\end{equation}
\begin{equation} \label{Rozovsky_ineq}
\Delta_n \le \CR(\eps, \gamma)\cdot\roz(g,\eps, \gamma),
\end{equation}
where 
$$
\begin{array}{rcll}
\CE(\eps, \,\cdot\,)&=&\essC(\eps, \,\cdot\,),\quad\ \eps\in(0,1],&\text{ in particular, } \CE(+0,\,\cdot\,)= \infty,
\\
\CE(\eps, \,\cdot\,)&\le&\essC(1, \,\cdot\,),\quad\ \eps>1;
\\
\CR(\eps,\,\cdot\,)&=&\rozC(\eps,\,\cdot\,),\quad\  \eps\in(0,1],&\text{ in particular, } \CR(+0,\,\cdot\,)= \infty,
\\
\CR(\eps,\,\cdot\,)&\le&\eps\rozC(\eps,\,\cdot\,),\quad \eps>1,
&\CR(\infty,\,\cdot\,) = \infty,
\end{array}
$$
and $\essC,$ $\rozC$ are as in~\eqref{Esseen2018}$,$ \eqref{Rozovskii2018}.
\end{theorem}

Observe that inequalities~\eqref{AW_ineq}, \eqref{Rozovsky_ineq} with $\gamma=1$ reduce to \eqref{EsseenAW2018}, \eqref{Rozovskii2019} with $\CE(\,\cdot\,)=\CE(\,\cdot\,,1)$, $\CR(\,\cdot\,)=\CR(\,\cdot\,,1)$, and with $g=g_*$, $\eps\in(0,1]$ to \eqref{Esseen2018}, \eqref{Rozovskii2018}, respectively. Moreover,~\eqref{AW_ineq} also improves Wang--Ahmad inequality~\eqref{AW_classic_ineq} due to moving the sum $\sum_{k=1}^n$ inside the modulus sign and under the least upper bound $\sup_{z>0}$ with the range becoming bounded to  $z<\eps B_n$, so that $A_6\le\CE(\infty,1)$. We call inequalities~\eqref{AW_ineq} and~\eqref{Rozovsky_ineq} analogues of  Esseen--Wang--Ahmad's and Rozovskii's inequalities. 

The next statement summarizes all what was said above on the constants $A_k,$ $k=\overline{1,6}$.

\begin{cor}\label{ThERWAineqConstRelations} 
We have
\begin{gather*}
0.5409<A_1=A_2\le1.87,
\\
0.5409<A_2\le A_4\le\essC(1,1)\le2.73,
\\
0.5409<A_1\le A_6\le\CE(\infty,1)\le\essC(1,1)\le2.73,
\\
A_3\le \min\{A_4,A_6,\essC(\infty,1)\}\le 2.66,
\\
A_5\le \rozC(1,1)\le 2.73,
\end{gather*}
with equalities $A_4=\essC(1,1),$ $A_6=\CE(\infty,1),$ $A_3=\essC(\infty,1),$  $A_5=\rozC(1,1)$ in the i.i.d.case.
\end{cor}

In~\cite{GabMakShe2020TVP} we complete corollary~\ref{ThERWAineqConstRelations} by showing, in particular, that
$$
A_3>0.3703,\quad A_5>0.5685.
$$

It is easy to see that the both fractions $L_{\bullet\,,\,n}\in\{\es,\roz\}$ are invariant with respect to scale transformations of a function $g\in\G$:
\begin{equation} \label{invariant}
L_{\bullet\,,\,n}(cg,\,\cdot\,,\,\cdot\,) = L_{\bullet\,,\,n}(g,\,\cdot\,,\,\cdot\,),
\quad c>0.
\end{equation}
Moreover, extremal properties of the functions 
$$
g_0(z):= B_ng_0(z,B_n)=\min\{z, B_n\}, \quad g_1(z):= B_ng_1(z,B_n)=\max\{z, B_n\},\quad z>0,
$$
in~\eqref{gmin_g_gmax_ineq} with $a:=B_n$ yield
\begin{equation}\label{gmin_gmax_ineq}
L_{\bullet\,,\,n}(g_0,\,\cdot\,,\,\cdot\,) \le L_{\bullet\,,\,n}(g,\,\cdot\,,\,\cdot\,) \le L_{\bullet\,,\,n}(g_1,\,\cdot\,,\,\cdot\,),
\end{equation}
hence the ``universal'' upper bounds for the appearing constants $\CE$ and $\CR$ in~\eqref{AW_ineq} and~\eqref{Rozovsky_ineq} are attained at $g=g_0$.
It is also obvious that with the extremal $g$ the both fractions $L_{\bullet\,,\,n}\in\{\es,\roz\}$ satisfy
\begin{equation}\label{identical_g}
L_{\bullet\,,\,n}(g_0,\eps,\,\cdot\,) = L_{\bullet\,,\,n}(g_*,\eps,\,\cdot\,)
\quad\text{for } \eps\le1,
\end{equation}
\begin{equation}\label{const_g}
L_{\bullet\,,\,n}(g_1,\eps,\,\cdot\,) = L_{\bullet\,,\,n}(g_{\textsc{c}},\eps,\,\cdot\,)
\quad\text{for } \eps\le1,
\end{equation}
where, as before, 
$$
g_*(z)=z,\quad g_{\textsc{c}}(z)\equiv1,\quad z>0.
$$

We also note that our proof of inequality~\eqref{AW_ineq} is completely  different from the one by Wang and Ahmad in~\cite{WangAhmad2016} who used a direct method based on a smoothing inequality and estimates for characteristic functions, similarly to the proof of Esseen's inequality~\eqref{Esseen_ineq} in~\cite{Esseen1969}. Our proof of~\eqref{AW_ineq}  is based on estimate~\eqref{Esseen2018} and property~\eqref{gmin_g_gmax_ineq} of the class $\G$ yielding inequality~\eqref{gmin_gmax_ineq} which makes our proof much simpler and shorter.

The next statement establishes some interesting properties or alternative expressions for the introduced fractions $\es(g,\eps,\gamma)$ and $\roz(g,\eps,\gamma)$.

\begin{theorem}\label{ThLn(g)Properties}
For all $\eps>0$ and $\gamma>0$ we have
\begin{equation}\label{Rozovsky_g0}
\roz(g_0,\eps, \gamma) = 
\frac{\gamma}{\eps\vee1} \abs{M_n(\eps)} + \sup_{0 < z < \eps\wedge1} z L_n(z),
\end{equation}
\begin{equation}\label{g_1_Esseen}
1 \le \es(g_1,\eps, \gamma) \le \max\{\eps, 1\}\cdot \max\{\gamma, 1\} ,
\end{equation}
\begin{equation}\label{g_1_Rozovsky}
1 \le \roz(g_1,\eps, \gamma) \le \max\{\eps, 1\} \cdot (\gamma + 1),
\end{equation}
in particular, 
$$
\es(g_1,\eps,\gamma)\equiv1\quad\text{for }\gamma\le1 \text{ and } \eps\le1,
$$
and, in the symmetric case, also
$$
\es(g_1,\eps, \gamma)\equiv\roz(g_1,\eps, \gamma)\equiv1\quad\text{for all } \gamma>0 \text{ and } \eps\le1.
$$
\end{theorem}

\section{Proofs}

The proof of theorem~\ref{ThAWanalog} uses~\eqref{Rozovsky_g0}, so, we start with the proof of theorem~\ref{ThLn(g)Properties}.

\begin{proof}[Proof of theorem~\ref{ThLn(g)Properties}]
Recall that $g_0(z)=\min\{z,B_n\}$, $g_1(z)=\max\{z,B_n\},$ $z>0$,
$$
\es(g,\eps, \gamma) = \sup_{0 < z < \eps} \frac{g(zB_n)}{zg(B_n)} \left( \gamma\abs{\essM(z)} + zL_n(z) \right),
$$
$$
\roz(g,\eps, \gamma) = \gamma\,\frac{g(\eps B_n)}{\eps g(B_n)} \abs{\essM(\eps)} + \sup_{0 < z < \eps}\frac{g(zB_n)}{g(B_n)}\,L_n(z),
$$
$$
\essM(z)=\frac1{B_n^3}\sum_{k=1}^n \mu_k(zB_n) = \frac1{B_n^3}\sum_{k=1}^n \E X_k^3\I(|X_k|<zB_n),
$$
$$
L_n(z) = \frac{1}{B_n^2} \sum_{k=1}^n \sigma_k^2(zB_n) = \frac{1}{B_n^2} \sum_{k=1}^n \E X_k^2\I (\abs{X_k} \ge zB_n).
$$

Representation~\eqref{Rozovsky_g0} is trivial for $\eps \le 1$ due to the observation that $g_0(z)=z$ for $z\le B_n$. As for $\eps > 1$, we have
$$
\roz(g_0,\eps, \gamma) =
\frac{\gamma}{\eps}\abs{\essM(\eps)} + \max\Big\{ \sup_{0 < z < 1} zL_n(z), \sup_{1\le z<\eps}L_n(z)\Big\}
$$
 for all $\gamma>0$. Since $L_n(z)$ is left-continuous and non-increasing, we have
$$
\sup_{1\le z<\eps}L_n(z)=L_n(1)\le\sup_{0 < z < 1} zL_n(z),
$$
and hence,
$$
\roz(g_0,\eps, \gamma) =\frac{\gamma}{\eps}\abs{\essM(\eps)} + \sup_{0<z<1}zL_n(z),
$$
which coincides with the right-hand side of~\eqref{Rozovsky_g0} for $\eps\ge1$.

To prove~\eqref{g_1_Esseen}, first, observe that~\eqref{Lambda(eps)<=eps} yields
\begin{equation}\label{Mn(eps)<=eps} 
\abs{M_n(\eps)}\le \frac{1}{B_n^3} \sum_{k=1}^n \E|X_k|^3\I(|X_k|<\eps B_n)=\osiM(\eps)\le\eps,\quad \eps>0.
\end{equation}
If $\eps \le 1$, then
$$
\es(g_1,\eps,\gamma)= \sup_{0 < z < \eps} \left\{ \frac{\gamma}{z}\abs{M_n(z)}+L_n(z)\right\}\le \max\{1,\gamma\}\sup_{0 < z < \eps} \left\{ \frac{\abs{M_n(z)}}{z}+L_n(z)\right\}
$$
$$
\le \frac{\max\{1, \gamma\}}{B_n^2}\sup_{0 < z < \eps} \left\{ \frac{1}{zB_n} \sum_{k=1}^n \E|X_k|^3\I(|X_k|<z B_n) + \sum_{k = 1}^n \E X_k^2\I(|X_k|\ge z B_n) \right\}
$$
$$
\le \max\{\gamma, 1\} = \max\{\eps, 1\} \cdot \max\{\gamma, 1 \},\quad \gamma>0.
$$
If $\eps > 1$, then
$$
\es(g_1,\eps, \gamma) = \max \left\{ \es(g_1,1, \gamma), \sup_{1\le z<\eps} \left\{ \gamma \abs{M_n(z)} + zL_n(z)\right\} \right\},\quad\gamma>0.
$$
As we have just seen, $\es(g_1,1, \gamma)\le \max\{\gamma, 1\}.$ And the second expression here can be bounded from above in the following way:
$$
\sup_{1\le z<\eps} \left\{ \gamma \abs{M_n(z)} + zL_n(z)\right\} 
$$
$$
\le  \frac{\max\{ \gamma, 1 \}}{B_n^3} \sup_{1\le z < \eps} \left\{\sum_{k=1}^n \E|X_k|^3\I(|X_k|\le zB_n) + zB_n\sum_{k=1}^n\sigma_k^2(z) \right\}
$$
$$
\le \max\{ \gamma, 1 \}\sup_{1\le z<\eps}z=
\eps \cdot \max\{\gamma, 1\},
$$
which completes the proof of the upper bound in~\eqref{g_1_Esseen}. To prove the lower bound in~\eqref{g_1_Esseen}, observe that for every $\eps>0$ and $\gamma>0$ we have
$$
\es(g_1,\eps, \gamma) \ge\lim_{z\to0+}\frac{g_1(zB_n)}{zg_1(B_n)} \left( \gamma\abs{\essM(z)} + zL_n(z) \right)= \lim_{z\to0+}\left(\frac{\gamma}z\abs{\essM(z)} + L_n(z) \right)
$$
$$
\ge\lim_{z\to0+}L_n(z)=1.
$$

Let us prove~\eqref{g_1_Rozovsky}. If $\eps \le 1$, then for all $\gamma>0$ we have
$$
\roz(g_1,\eps, \gamma) = \frac\gamma\eps\abs{\essM(\eps)} + \sup_{0<z<\eps}L_n(z) =\frac\gamma\eps\abs{\essM(\eps)} + L_n(0)=\frac\gamma\eps\abs{\essM(\eps)} +1,
$$
which trivially yields the lower bound $\roz(g_1,\eps, \gamma)\ge1$ and, with~\eqref{Mn(eps)<=eps}, also the upper bound $\roz(g_1,\eps, \gamma) \le\gamma+1$. Combining the two-sided bounds, we obtain~\eqref{g_1_Rozovsky} for $\eps\le1$.  If $\eps > 1$, then
$$
\roz(g_1,\eps, \gamma) = \gamma \abs{M_n(\eps)} + \max\Big\{L_n(0), \sup_{1\le z<\eps} z L_n(z)\Big\},
$$
whence trivially a lower bound 
$$
\roz(g_1,\eps, \gamma)\ge \max\Big\{L_n(0), \sup_{1\le z<\eps} z L_n(z)\Big\}\ge L_n(0)=1
$$
follows. Furthermore, using $\sup\limits_{1\le z<\eps} z L_n(z)\le\eps\sup\limits_{1\le z<\eps}L_n(z)=\eps L_n(0)=\eps$ and also~\eqref{Mn(eps)<=eps}, we obtain an upper bound
$$
\roz(g_1,\eps, \gamma)\le  \gamma \eps +  \max\{1,\eps\}= \eps(\gamma + 1),
$$
which proves~\eqref{g_1_Rozovsky} also for $\eps>1$.

The concluding remarks follow from the observation that $M_n(z)\equiv0$ for $z\ge0$ in the symmetric case, and hence, the fractions $\es(g,\eps,\gamma)$, $\roz(g,\eps,\gamma)$ are constant with respect to $\gamma>0$. Letting $\gamma\to0+$, we obtain two-sided bounds 
$$
1\le L_{\bullet\,,\,n}(g_1,\eps,\gamma)\le \max\{\eps,1\},\quad \gamma>0,
$$
for both fractions $L_{\bullet\,,\,n}\in\{\es,\roz\}$, whence it follows that  $\es(g_1,\eps,\gamma)=\roz(g_1,\eps,\gamma)\equiv1$ for $\eps \le 1$ and $\gamma>0$.
\end{proof}

\begin{proof}[Proof of theorem~\ref{ThAWanalog}]
Recall that $g_*(z) \equiv z\in\G$. Let us fix any $g \in \G, n, F_1, F_2, \ldots, F_n$ and consider two cases.

If $\eps \le 1$, then, due to~\eqref{identical_g} and~\eqref{gmin_gmax_ineq}, the fractions $L_{\bullet\,,\,n}\in\{\es,\roz\}$ in~\eqref{Esseen2018}, \eqref{Rozovskii2018}, \eqref{AW_ineq}, \eqref{Rozovsky_ineq} satisfy
$$
L_{\bullet\,,\,n}(\eps,\,\cdot\,) =L_{\bullet\,,\,n}(g_*,\eps,\,\cdot\,) =L_{\bullet\,,\,n}(g_0,\eps,\,\cdot\,) \le  L_{\bullet\,,\,n}(g,\eps,\,\cdot\,),\quad g\in\G.
$$
Using this inequality in~\eqref{Esseen2018} and \eqref{Rozovskii2018} we obtain
$$
\Delta_n \le \essC(\eps, \,\cdot\,) \cdot \es(\eps, \,\cdot\,)\le  \essC(\eps, \,\cdot\,) \cdot \es(g,\eps, \,\cdot\,),
$$
$$
\Delta_n \le \rozC(\eps, \,\cdot\,) \cdot \es(\eps, \,\cdot\,)\le  \rozC(\eps, \,\cdot\,) \cdot \es(g,\eps, \,\cdot\,) 
$$
for all $g\in\G$, which yields~\eqref{AW_ineq}, \eqref{Rozovsky_ineq} with $\CE(\eps,\,\cdot\,)\le\essC(\eps,\,\cdot\,)$ and $\CR(\eps,\,\cdot\,)\le\rozC(\eps,\,\cdot\,)$. Observing that $g_0\in\G$ we conclude that inequalities are identities, in fact.

Let $\eps > 1$. Since $\es(g,\eps,\gamma)$ is non-decreasing with respect to $\eps$, we have
$$
\Delta_n \le \essC(1, \,\cdot\,) \es(g, 1,\,\cdot\,) \le \essC(1, \,\cdot\,)  \es(g, \eps,\,\cdot\,) ,
$$
for all $g\in\G$, which yields~\eqref{AW_ineq} with $\CE(\eps,\,\cdot\,)\le\essC(1,\,\cdot\,)$ for all $\eps>1$. 
Furthermore, using~\eqref{Rozovsky_g0} we obtain
$$
\roz(g_0,\eps, \gamma) = 
\frac{\gamma}{\eps} \abs{M_n(\eps)} + \sup_{0<z<1} z L_n(z).
$$
The second term here, due to the monotonicity of $L_n(z)$, can be bounded from below as
$$
\sup_{0<z<1}zL_n(z)=\sup_{0<z<\eps}\frac{z}{\eps}L_n\left(\frac{z}\eps\right)\ge \frac1\eps\sup_{0<z<\eps}zL_n(z),
$$
hence, we obtain a lower bound
$$
\roz(g_0,\eps, \gamma) \ge
\frac{1}{\eps} \left (\gamma\abs{M_n(\eps)} + \sup_{0<z<\eps} z L_n(z)\right) =\frac1\eps\,\roz(g_*,\eps,\gamma).
$$ 
Now observing that
$$
\roz(\eps,\,\cdot\,) =\roz(g_*,\eps,\,\cdot\,) \le \eps\roz(g_0,\eps,\,\cdot\,) \le \eps\roz(g,\eps,\,\cdot\,),\quad g\in\G,
$$
and using~\eqref{Rozovskii2018}, we obtain~\eqref{Rozovsky_ineq} with $\CR(\eps,\,\cdot\,)\le\eps\rozC(\eps,\,\cdot\,)$ for $\eps\ge1$.

The fact that $\CE(+0, \,\cdot\,) = \CR(+0, \,\cdot\,) = +\infty$ for all $\gamma>0$ trivially follows either from unboundedness of $\essC(+0, \,\cdot\,)$ and $\rozC(+0, \,\cdot\,)$, or, directly, from~\eqref{gmin_gmax_ineq} and
$$
\lim_{\eps\to+0}\es(g_0,\eps,\gamma) = \lim_{\eps\to+0}\roz(g_0,\eps,\gamma)=0, \quad  \gamma > 0.
$$
To prove that $\lim\limits_{\eps \to\infty} \CR(\eps,\,\cdot\,) = +\infty$, assume that the random summands $X_1,\ldots,X_n$ are i.i.d. and have finite third-order moments. Then, due to~\eqref{Rozovsky_g0}, we have
$$
\lim\limits_{\eps \to\infty}\roz(g_0,\eps, \gamma) = 
\lim\limits_{\eps \to\infty}\frac{\gamma}{\eps} \abs{M_n(\eps)} + \sup_{0<z<1} zL_n(z) =\sup_{0<z<1} zL_n(z).
$$
On the other hand, in~\cite[Theorem~3]{GabdullinMakarenkoShevtsova2019-2} it is shown that 
$$
\sup_{F} \limsup_{n \rightarrow +\infty} \frac{\Delta_n}{\sup\limits_{0<z<1} zL_n(z)} = +\infty,
$$
where  the least upper bound is taken over all identical distribution functions $F_1=\ldots=F_n=F$ of the random summands $X_1,\ldots,X_n$ with finite third-order moments. Since estimate~\eqref{Rozovsky_ineq} must hold also in this particular case, we should necessarily have ${\CR(\eps,\,\cdot\,)\to\infty}$ as ${\eps\to\infty}$.
\end{proof}


\end{document}